\documentclass{amsart}
\usepackage[utf8]{inputenc}
\usepackage{amsmath, amsthm, amssymb, amsfonts, mathtools, stmaryrd, latexsym, enumerate}

\pdfoutput=1

\usepackage[dvipsnames]{xcolor}
\usepackage{float}
\usepackage{enumitem}
\usepackage{tikz}
\usepackage{comment}
\usepackage[linesnumbered,lined,commentsnumbered]{algorithm2e}

\usepackage{listings}
\usepackage{graphicx}
\usepackage{caption}
\usepackage[all]{xypic}
\usepackage{verbatim}
\usepackage{subcaption}

\usepackage[colorlinks]{hyperref}
\hypersetup{
    colorlinks,
    linkcolor={BrickRed},
    citecolor={ForestGreen},
    urlcolor={NavyBlue}
}
\usepackage[nameinlink]{cleveref}

\usepackage[parfill]{parskip}
\usepackage[margin=1in]{geometry}

\theoremstyle{plain}
\newtheorem{theorem}{Theorem}
\numberwithin{theorem}{section}
\newtheorem{proposition}[theorem]{Proposition}
\newtheorem{lemma}[theorem]{Lemma}
\newtheorem{corollary}[theorem]{Corollary}
\newtheorem{definition}[theorem]{Definition}

\theoremstyle{definition}
\newtheorem{remark}[theorem]{Remark}
\newtheorem{example}[theorem]{Example}

\newtheorem{question}[theorem]{Question}
\newtheorem*{notation}{Notation}
\newtheorem*{claim}{Claim}

\newcommand{\bbC}{\mathbb{C}}
\newcommand{\bbG}{\mathbb{G}}
\newcommand{\bbP}{\mathbb{P}}

\newcommand{\calQ}{\mathcal{Q}}
\newcommand{\calO}{\mathcal{O}}

\title{Algebraic surfaces as Hadamard products of curves}

\author{Dario Antolini}
\address{%
	Dario Antolini\newline
	Dipartimento di Matematica, Università di Trento\newline
	Email: \href{mailto:dario.antolini-1@unitn.it}{dario.antolini-1@unitn.it}
}

\author{Edoardo Ballico}
\address{%
	Edoardo Ballico\newline
	Dipartimento di Matematica, Università di Trento\newline
	Email: \href{mailto:edoardo.ballico@unitn.it}{edoardo.ballico@unitn.it}
}

\author{Alessandro Oneto}
\address{%
	Alessandro Oneto\newline
	Dipartimento di Matematica, Università di Genova\newline
	Email: \href{mailto:alessandro.oneto@unige.it}{alessandro.oneto@unige.it}
}

\thanks{This paper is published in Journal of Algebra 697:725-735 (2026).  \href{https://doi.org/10.1016/j.jalgebra.2026.03.008}{doi:10.1016/j.jalgebra.2026.03.008}.}

	
\begin{document}

\begin{abstract} 
    We study projective surfaces in $\bbP^3$ which can be written as Hadamard product of two curves. We show that quadratic surfaces which are Hadamard product of two lines are smooth and tangent to all coordinate planes, and such tangency points uniquely identify the quadric. The variety of such quadratic surfaces corresponds to the Zariski closure of the space of symmetric matrices whose inverse has null diagonal. For higher-degree surfaces which are Hadamard product of a line and a curve we show that the intersection with the coordinate planes is always non-transversal.
\end{abstract}


\maketitle

\section{Introduction}
Motivated by the study of algebraic statistical models called {\it Restricted Boltzmann Machines}, the definition of {\it Hadamard product} of algebraic varieties has been introduced in \cite{cueto2010geometry}. This work was followed by several papers which continued the study of Restricted Boltzmann Machines and related models by means of algebraic and tropical geometry; see e.g. \cite{cueto2010implicitization,montufar2015discrete,montufar2018restricted,seigal2018mixtures,oneto2023hadamard}. At the same time, Hadamard products of algebraic varieties have been studied from a purely algebro-geometric perspective following several directions. In \cite{bocci2016hadamard,bocci2017hadamard}, the authors considered Hadamard products of linear spaces, while in \cite{bocci2018hilbert} properties of Hilbert functions of Hadamard products were investigated. In \cite{carlini2019starconfig}, the authors generalized a construction of \cite{bocci2016hadamard} to describe the construction of star configurations through Hadamard products. In \cite{bocci2023gorenstein}, Hadamard products were used to construct Gorenstein sets of points in $3$-dimensional projective space. In \cite{bocci2022hadamard}, Hadamard products of hypersurfaces were considered. A uniform presentation of the theory of Hadamard products of algebraic varieties can be found in \cite{bocci2024hadamard}.

This paper fits into this line of research by looking for characterizations of algebraic varieties which are Hadamard products of smaller varieties. In particular, we begin to address the following question: {\it when can an algebraic surface in $3$-dimensional projective space be written as an Hadamard product of two curves?} 

In \cite{ballico2025projective}, the second author showed that very general surfaces of degree at least $4$ cannot be written as Hadamard product of two curves.

In \Cref{sec:quadratic} we investigate properties of quadratic surfaces in $3$-dimensional projective space which are Hadamard products of two lines. We observe that these are always smooth and tangent to the four coordinate planes. This allows us to deduce a complete characterization of the space of such quadrics.

\begin{theorem}
\label{theorem:quadraticsurfaces}
    The variety of quadratic surfaces that are Hadamard products of two lines is 5-dimensional and it corresponds to the variety of quadrics defined by inverses of symmetric matrices with null diagonal.
\end{theorem}

Moreover, given a quadratic surface which is Hadamard product of lines, we investigate the geometry of the four singular points of the four reducible conics obtained as intersection of the quadric with the coordinate planes. In \Cref{theorem:identifiability}, we notice that such points uniquely determine the quadratic surface.

In \Cref{sec:higher_degrees}, we approach the case of higher-degree surfaces. In \Cref{cor:non-trasversal}, we provide an extension of \Cref{theorem:quadraticsurfaces} by proving that if a surface $W$ of degree at least $3$ is the Hadamard product of a line and a curve, then the intersection of $W$ with the coordinate planes is not transversal. 

The codes used for some of the proofs and examples of the paper were written in the language of Macaulay2 \cite{M2} and have been collected in \cite{datahadamardquadric}.

\subsection*{Funding} 
All authors thank the TensorDec Laboratory of the Department of Mathematics of the University of Trento for useful discussions and offering a fruitful research environment. All authors are members of INdAM-GNSAGA. AO has been partially supported by MUR (Ministero dell'Università e della Ricerca, Italy) through the PRIN2022SC project Prot.\ 2022NBN7TL (CUP: E53C24\-002320006) ``Applied Algebraic Geometry of Tensors''.

\section{Hadamard product}
\label{sec:definitions}

We first recall the basic definitions. For a more extensive presentation, see \cite{bocci2024hadamard}. Let $V$ be an $(N+1)$-dimensional $\bbC$-vector space on which we fix a basis $e_0, \ldots, e_N$. 

\begin{definition}
    The \emph{Hadamard product} over the projective space $\mathbb{P}^N = \bbP V$ is the coordinate-wise multiplication, i.e., given $p = (p_0:\cdots:p_N)$ and $q = (q_0:\cdots:q_N)$, then $p \star q := (p_0q_0:\cdots:p_Nq_N).$ 
\end{definition}

The Hadamard product is not well-defined everywhere but it defines a rational map $h\colon \bbP^N \times \bbP^N \dashrightarrow \bbP^N$. Such rational map corresponds to the composition of the Segre embedding of $\bbP^N \times \bbP^N$ in $\bbP^{(N+1)^2-1}$ with the linear projection onto the {\it diagonal coordinates}. 

\begin{definition}
    Given two algebraic varieties $X,Y \subset \bbP^N$, their {\it Hadamard product} is the Zariski closure of the image of $X \times Y$ through the Hadamard map, i.e.,
    \begin{equation}
    \label{def:HadamardProduct}
        X \star Y := \overline{h(X \times Y)} = \overline{\{p \star q ~:~ p \in X, \ q \in Y, \ p \star q \text{ exists}\}.}
    \end{equation}
    In particular, given an algebraic variety $X \subset \bbP^N$, we define its {\it Hadamard powers} iteratively as:
    \[
        X^{\star 0} := (1:\cdots:1) \quad \text{ and } \quad X^{\star s} := X \star X^{\star (s-1)}.
    \]
\end{definition}

\begin{remark}
\label{rmk:hadamardirreducible}
    If $X, Y \subset \bbP^N$ are irreducible, then $X \star Y \subset \bbP^N$ is irreducible, since it is (the closure of) the image of a rational map.
\end{remark}

\begin{remark}
\label{remark:stabilizer}
    The subgroup $D_N \subset \text{Aut}(\bbP^N)$ of invertible diagonal matrices, which is isomorphic to the $N$-dimensional torus $\{ (x_0 : \cdots : x_N) \in \bbP^N : \prod_i x_i \neq 0 \}$ of $\bbP^N$, preserves the Hadamard product, in the sense that $X \star Y = g(X) \star g^{-1}(Y)$ for all $g \in D_N$.
\end{remark}


\section{Quadratic surfaces as Hadamard products of lines}
\label{sec:quadratic}

\subsection{Smoothness and tangencies: a complete characterization}
\label{sec:smoothquadratic}

In this section we give a complete characterization of quadratic surfaces that are Hadamard products of lines by proving \Cref{theorem:quadraticsurfaces}. Before showing the proof, let us recall a few basic facts on Hadamard products of lines.

\begin{notation}
    For every $i = 0,\ldots,N$, let $H_i = \{x_i = 0 \}$ be the $i$-th coordinate hyperplane in $\bbP^N$. We will call {\em coordinate point} (resp.\ {\em line}, {\em plane}) any intersection of $N$ (resp.\ $N-1$, $N-2$) coordinate hyperplanes. For any $p_1,\ldots, p_m \in \bbP^N$, $\langle p_1,\ldots, p_m\rangle$ denotes their linear span in $\bbP^N$.
\end{notation}

\begin{remark}
\label{rmk:hadamard_product_lines}

Let $L, R \subset \bbP^3$ be lines.

\begin{enumerate}
    \item If $L = R = \langle p,q \rangle$, then $L^{*2}$ is contained in a plane, in particular it cannot be a quadratic surface. Indeed, either $L = H_i \cap H_j$ is a coordinate line, hence $L^{\star 2} = L$ or $p$ and $q$ can be chosen to be not coordinate points and $L^{\star 2} \subset \langle p \star L, q \star L\rangle$ where the latter is at most $2$-dimensional because $p \star L$ and $q \star L$ are at most lines with non-trivial intersection since $p \star q \in (p \star L) \cap (q \star L)$.

    \item If $L$ or $R$ is contained in a coordinate plane, then $L \star R$ is contained in the same coordinate plane, in particular it cannot be a quadratic surface. 
\end{enumerate}
\end{remark}

\begin{lemma}
    \label{lemma:hadamardproduct_smooth}
    Let $L,R\subset\bbP^3$ be two lines such that $W = L\star R$ is a quadratic surface. Then, $W$ is smooth and it is tangent to all four coordinate planes.
\end{lemma}
\begin{proof}
    By assumption, $L \star R$ is not contained in any coordinate plane: in particular, we deduce that $L$ and $R$ are not contained in any coordinate plane and $L \neq R$. Consider the restriction $h_{L,R} \colon L \times R \dashrightarrow W$ of the Hadamard map on $\bbP^3\times\bbP^3$. A priori, this is only a rational map; however, since it is defined by the $3$-dimensional complete linear system $|\mathcal{O}_{L\times R}(1,1)|$, $h_{L,R}$ is actually a morphism. Note that $h_{L,R}$ is finite. In fact, suppose by contradiction that $h_{L,R}$ is not finite. Hence, since $h_{L,R}$ does not map $L \times R$ to a single point, we have that $h_{L,R}$ contracts a curve $C \subset L \times R$. However, by Grauert's criterion \cite[Proposition III.2.1]{barth2015compact}, this implies that $C$ is a curve with negative self-intersection, which is a contradiction because $L \times R \simeq \bbP^1 \times \bbP^1$ contains no such curves. Moreover, since by assumption $\deg(L\times R) = \deg(W)$, $h_{L,R}$ is also birational. 
    
    Recall that $W$ is always normal: even if $W$ was a singular cone, then it is normal by a criterion of Serre, see \cite[Theorem II.8.22A]{hartshorne1977algebraic}. Therefore, since $h_{L,R}$ is a birational morphism with target a normal variety, then, by Zariski's Main Theorem \cite[Corollary III.11.4]{hartshorne1977algebraic}, $h_{L,R}$ is actually an isomorphism and $W$ is smooth.

    Consider now the intersections $p_i = L \cap H_i$ for all $i = 0,\ldots,3$. Then, $p_i \star R$ is a line which is contained in $W \cap H_i$. Since $W \cap H_i$ contains a line, it is a reducible conic for all $i = 0,\ldots,3$ and by this we conclude. 
\end{proof}

The latter characterization of quadratic surfaces which are Hadamard products of lines allows us to prove \Cref{theorem:quadraticsurfaces}.

\begin{notation}
    We denote by $\bbG(1,3)$ the Grassmannian of lines in $\bbP^3$.
\end{notation}

\begin{proof}[Proof of \Cref{theorem:quadraticsurfaces}]
    Inside the $9$-dimensional space of quadratic surfaces in $\bbP^3$, let 
    \[
        H = \overline{\{W = L \star R ~:~ L,R \text{ lines}, \ W \text{ quadratic surface}\}}.
    \]
    In other words, $H$ is the closure of the image of the rational map $\varphi \colon \bbG(1,3)\times\bbG(1,3) \dashrightarrow H$ which sends a generic pair of lines to their Hadamard product.

    The generic fiber of $\varphi$ is at least 3-dimensional: indeed, $\varphi^{-1}(L\star R)$ contains all pairs $(g(L),g^{-1}(R))$ for all $g \in D_3$ as in \Cref{remark:stabilizer}. 

    \begin{claim}
        Let $\calQ = \{x_0x_3 - x_1x_2 = 0\}$ be the Segre quadric in $\bbP^3$. Then, $\varphi^{-1}(\calQ)$ is $3$-dimensional.
    \end{claim}

    Before proving the Claim, which is just a straightforward computation, we see how we conclude the proof from that. If the Claim holds true, then the generic fiber is exactly $3$-dimensional. Since $\dim (\bbG(1,3)\times\bbG(1,3)) = 8$, we deduce that $\dim(H) = 8 - 3 = 5$.
    
    By \Cref{lemma:hadamardproduct_smooth}, we know that $H$ is contained in the variety of quadrics that are tangent to the coordinate planes:
    \[
        Y = \overline{\{Q ~:~ Q \text{ is tangent to }H_i \text{ for any } i = 0,1,2,3\}}.
    \]
    If $A$ is the non-singular symmetric matrix defining $Q \in Y$, then $A^{-1}$ has null-diagonal: indeed, if $Q$ is tangent to the $H_i$'s then the dual quadric $Q^\vee$ passes through the coordinate points and it is well-known that $Q^\vee$ is defined by $A^{-1}$ \cite[Example I.2.1]{gelfand1994discriminants}. In other words, $Y$ is parametrized by the space of $4 \times 4$ non-singular symmetric matrices with null-diagonal, which is irreducible and $5$-dimensional through the map $A \mapsto A^{-1}$: in particular, $Y$ is $5$-dimensional and irreducible. Therefore, $H = Y$.
    
    \begin{proof}[Proof of Claim]
        Let $(L,R) \in \bbG(1,3)\times\bbG(1,3)$. Without loss of generality, we may assume that they are represented by the matrices
        \[
            L = \begin{bmatrix}
                1 & 0 & u_1 & v_1\\
                0 & 1 & u_2 & v_2
            \end{bmatrix}, \quad
            R = \begin{bmatrix}
                1 & 0 & z_1 & w_1\\
                0 & 1 & z_2 & w_2
            \end{bmatrix}.
        \]
        In other words, we consider the restriction of $\varphi$ to an affine chart of $\bbG(1,3)\times\bbG(1,3)$. Then, if $p \in L, q \in R$, we have that:
        \[
        p = (\lambda_1: \lambda_2: \lambda_1 u_1 + \lambda_2 u_2: \lambda_1 v_1 + \lambda_2 v_2), \ q = (\mu_1: \mu_2: \mu_1 z_1 + \mu_2 z_2: \mu_1 w_1 + \mu_2 w_2).    
        \]
        for some $\lambda_1, \lambda_2, \mu_1, \mu_2 \in \bbC$. Hence:
        \[
            p \star q = (\lambda_1\mu_1: \lambda_2\mu_2: (\lambda_1 u_1 + \lambda_2 u_2)(\mu_1 z_1 + \mu_2 z_2): (\lambda_1 v_1 + \lambda_2 v_2)(\mu_1 w_1 + \mu_2 w_2)).
        \]
        Now, since $L \star R = \calQ$, we have that the polynomial obtained by plugging in the equation 
        \begin{multline*}
            \lambda_1^2\mu_1^2(v_1w_1) + \lambda_1^2\mu_1\mu_2(v_1w_2) + \lambda_1\lambda_2\mu_1^2(v_2w_1) + \lambda_1\lambda_2\mu_2\mu_2(v_2w_2 - u_2z_2)\\
            - \lambda_1\lambda_2\mu_2^2(u_1z_2) - \lambda_2^2\mu_1\mu_2(u_2z_1) - \lambda_2^2\mu_2^2(u_2z_2) \in \bbC[u_1,u_2,v_1,v_2,w_1,w_2,z_1,z_2][\lambda_1,\lambda_2,\mu_1,\mu_2]
        \end{multline*}
        is identically zero, i.e., $v_1w_1 = v_1w_2 = v_2w_1 = v_2w_2 - u_2z_2 = u_1z_2 = u_2z_1 = u_2z_2 = 0$. This defines a $3$-dimensional subspace of $\bbG(1,3)\times\bbG(1,3)$.
    \end{proof}
    Thus the claim is proved and this concludes the proof of \Cref{theorem:quadraticsurfaces}.
\end{proof}

In \Cref{lemma:hadamardproduct_smooth}, we have seen that quadratic surfaces which are Hadamard products of lines are smooth. In the following example, we exhibit a quadratic cone which is Hadamard product of a line and a conic. We will say more about surfaces as Hadamard products of higher degree curves in \Cref{sec:higher_degrees}.

\begin{example}
    Let $L\subset \bbP^3$ be a generic line through $(1:0:0:0)$ and $C\subset \bbP^3$ be a generic (plane) conic through $(0:1:1:1)$. The result is a quadratic cone singular at $(1:0:0:0)$. Direct computation were done with the support of the algebra software Macaulay2 \cite{M2}, see \cite{datahadamardquadric} for the code.
\end{example}

In the following example we show that the Hadamard product of a pair of skew lines does not necessarily give a quadratic surface.

\begin{example}
    Let $L \subset \bbP^3$ be a line which does not intersect any coordinate line. Hence, $L^{\star 2}$ is a plane as in \Cref{rmk:hadamard_product_lines}. It follows that, for each $p \neq (1:1:1:1)$ with no zero coordinate, $L \star L_p = p \star L^{\star 2}$ is a plane, where $L_p = p \star L$. This gives a family of lines $L, R \subset \bbP^3$ such that $R = p \star L \neq L$ and $L \star R$ is a plane. Moreover, for a generic line $L \subset \bbP^3$, the family $X_L =\{p \star L:p \in \bbP^3, \ p \star L \in \bbG(1,3)\}$ is a 3-dimensional quasi-projective subvariety of $\bbG(1,3)$. Indeed, it is the image of the rational map $\varphi_L\colon \bbP^3 \dashrightarrow \bbG(1,3)$ defined by $\varphi_L(p) = p \star L$. Denoting by $(q_{12}:\cdots : q_{34})$ the Plücker coordinates of $L$, the locus $\overline{X_L}\subset \bbG(1,3)$ is cut out by the following equation in the Plücker coordinates $(p_{12}:\cdots:p_{34}) \in \bbG(1,3)$:
    \[
        q_{12}q_{34}p_{13}p_{24} - q_{13}q_{24}p_{12}p_{34} =0
    \]
    which we obtained by elimination. For an example of two incident lines giving a plane, see \cite[Example~8.2.3]{bocci2024hadamard} or \cite[Example 2.10]{bocci2018hilbert}.
\end{example}

\subsection{Coordinate loci: an identifiability result}
\label{sec:coordinateloci}
In view of \Cref{theorem:quadraticsurfaces}, in this section we focus on the singular points of the conics obtained by intersecting a quadratic surface with the coordinate hyperplanes.

\begin{remark}
\label{rmk:tangenthypersurfaces}
    If $L, R \subset \bbP^3$ are lines and $W = L \star R$ is a quadratic surface, then by \Cref{lemma:hadamardproduct_smooth} we know that $W$ is tangent to all the coordinate planes, hence $\text{Sing}(W \cap H_i)$ is a single point for all $i = 0, \ldots, 3$; we denote it by $O_i \in \bbP^3$.
\end{remark}

\begin{definition}
    Let $W \subset \bbP^3$ as in \Cref{rmk:tangenthypersurfaces}. The {\em singular coordinate locus} of $W \subset \bbP^3$ is defined as the set $\text{SCL}(W) = \{ O_0, \ldots,O_3\}$.
\end{definition}

We now see how such points identify the quadratic surface.

\begin{lemma}
\label{lemma:sclhas4points}
    Let $W \subset \bbP^3$ be a quadratic surface of the form $W = L \star R$, for some lines $L, R \subset \bbP^3$. Then, $|\text{SCL}(W)| = 4$.
\end{lemma}
\begin{proof}
    By \Cref{lemma:hadamardproduct_smooth}, we know that $W$ is smooth and tangent to all coordinate planes. As in \Cref{rmk:tangenthypersurfaces}, we have that $\text{SCL}(W) = \{ O_0, \ldots , O_3\}$. Suppose by contradiction that $|\text{SCL}(W)| < 4$. Without loss of generality, $O_0 = O_1$. Since $W\cap H_0$ and $W \cap H_1$ are singular conics with common singular point $O_0$ such that their irreducible components are the lines of the 2 rulings of $W$ passing through $O_0$, we get that $W \cap H_0 = W \cap H_1$. Now, $W \cap H_0 = W \cap H_1$ is 1-dimensional and contained in the coordinate line $H_0 \cap H_1$. Hence, $W \cap H_1 = W \cap H_0$ would be scheme-theoretically equal to the double coordinate line $H_0 \cap H_1$, which leads to a contradiction since $W$ is smooth quadric and in particular not a cone by \Cref{lemma:hadamardproduct_smooth}.
\end{proof}

Singular coordinate loci provide identifiability of quadratic surfaces which are Hadamard products of lines. 

\begin{lemma}
\label{lemma:identifiability}
    Given two sets of collinear points $\{p_0,\ldots,p_3\}$ and $\{q_0,\ldots,q_3\}$ in $\bbP^3$ with $p_i, q_i \in H_i \smallsetminus \cup_{j\neq i} H_j$, there is a unique quadratic surface $W \subset \bbP^3$ such that $W \cap H_i$ is singular at $p_i \star q_i$ for all $i = 0, \ldots, 3$. 
\end{lemma}
\begin{proof}
    We can generically write $p_0 = (0:1:a_1:a_2), \ p_1 = (1:0:a_3:a_4), \ q_0 = (0:1:b_1:b_2), \ q_1 = (1:0:b_3:b_4)$, with $\prod_ia_ib_i \neq 0$. Then, $p_2 = \langle p_0,p_1 \rangle \cap H_2 = (-a_1:a_3:0:a_2a_3-a_1a_4)$ and $p_3 = \langle p_0,p_1 \rangle \cap H_3 = (-a_2:a_4:-a_2a_3+a_1a_4:0)$ (analogously for $q_2,q_3$). Let $W = \{ F = 0\}$ be a quadratic surface with $F = c_0x_0^2 + \cdots + c_9x_3^2$ with respect to the lexicographical order of the homogeneous monomials. Imposing the singularity of $W_i = W\cap H_i = \{ F_i = 0\}$ at $O_i = p_i \star q_i$ for every $i = 0,\ldots, 3$ amounts to the system of linear equations in $\bbC[a_1,\ldots,a_4,b_1, \ldots, b_4][c_0, \ldots, c_9]$ given by $\partial F_i/\partial x_j(O_i) = 0$ for all $i, j = 0, \ldots, 3$ such that $j \neq i$. By direct computation with the support of the algebra software Macaulay2 \cite{M2}, the associated $12 \times 10$ matrix is never full-rank. The radical ideal of the $9\times 9$ minors has primary decomposition
    \begin{align*}
        (a_1,a_2) & \cap (a_1,a_3) \cap (a_2,a_4) \cap (a_3,a_4) \cap (b_1,b_2) \cap (b_1,b_3) \cap (b_2,b_4) \cap (b_3,b_4) \\
        & \cap (a_4,b_1) \cap (a_3, b_2) \cap (a_2,b_3) \cap (a_1,b_4) \cap (a_1,a_4,b_2,b_3) \cap (a_2,a_3,b_1,b_4).
    \end{align*}
    The code and the printed version of the Jupyter Notebook with all computations are available at \cite{datahadamardquadric}. 
    Since, as we already observed, by assumption $\prod_i a_ib_i \neq 0$, we deduce that the rank of the matrix is always $9$. In other words, for any $a_i$'s and $b_i$'s such that $\prod_i a_ib_i \neq 0$ the linear system has a unique solution up to scalars and this concludes the proof.
\end{proof}

\begin{theorem}
\label{theorem:identifiability}
    A quadratic surface which is the Hadamard product of two lines is uniquely determined by its singular coordinate locus, i.e., if $W = L \star R$ and $W' = L' \star R'$ such that $\text{SCL}(W)=\text{SCL}(W')$ then $W = W'$.
\end{theorem}
\begin{proof}
    Let $W = L \star R$ be a quadratic surface. Consider the reducible conics $W_i = W \cap H_i = \ell_i \cup r_i$ for $i = 0,\ldots,3$ and $O_i = \ell_i \cap r_i$ their centers. Recall that: $O_i = p_i \star q_i$ with $p_i = L \cap H_i$ and $q_i = R \cap H_i$ since $h_{L,R}$ is a morphism, by \Cref{lemma:hadamardproduct_smooth}, and $O_i \in H_i \smallsetminus \cup_{j \neq i} H_j$ by \Cref{lemma:sclhas4points}. In other words, the points $\{p_0,\ldots,p_3\}$ and $\{q_0,\ldots,q_3\}$ satisfy the assumption of \Cref{lemma:identifiability} and this concludes the proof.
\end{proof}

\begin{remark}
\label{rmk:coplanarity}
    For sake of completeness, we observe that the four points giving the singular coordinate locus of a quadratic surface which is Hadamard product of two lines might be coplanar. Indeed, the construction described in the proof of \Cref{theorem:identifiability} works also under the additional Zariski-closed assumption on the $a_i$'s and $b_i$'s obtained by imposing the coplanarity of the $O_i$'s.
\end{remark}

\begin{example}
    We construct families of Hadamard products of lines with coplanar singular coordinate points. Let $a, b \in \bbC \smallsetminus\{0\}$ and define:
    \[
        P_0 = (0:1:1:2), \quad P_1 = (1:0:1:1), \quad P_2 = (1:-1:0:-1), \quad P_3 = (2:-1:1:0),
    \]
    \[
        Q_0 = (0:1:1:1), \ Q_{1ab} = (1:0:a:b), \ Q_{2ab} = (1:-a:0:b-a), \ Q_{3ab} = (1:-b:a-b:0).
    \]
    Set $L = \langle P_0, P_1\rangle$, $R_{ab} = \langle Q_0, Q_{1ab} \rangle$ and $W_{ab} = L \star R_{ab}$. It is easy to check that $P_2, P_3 \in L$ and $Q_{2ab}, Q_{3ab} \in R_{ab}$. Moreover, consider the points:
    \[
        O_0 = (0:1:1:2), \ O_{1ab} = (1:0:a:b), \ O_{2ab} = (1:a:0:a-b), \ O_{3ab} = (2:b:a-b:0),
    \]
    where $O_0 = P_0 \star Q_0$ and $O_{iab} = P_{i} \star Q_{iab}$ for all $i = 1,2,3$. Then, $\text{SCL}(W_{ab}) = \{O_0, O_{1ab}, O_{2ab}, O_{3ab}\}$. As mentioned in \Cref{rmk:coplanarity}, the singular coordinate points can be coplanar by imposing the condition that
    \[
        \Delta_{ab} = \det\begin{bmatrix}
            0 & 1 & 1 & 2\\
            1 & 0 & a & b\\
            1 & a & 0 & a-b\\
            2 & b & b -a & 0
        \end{bmatrix} = 0.
    \]
    That is $\Delta_{ab} = a(3a-2b) = 0$. As described in the proof of \Cref{theorem:identifiability}, for any $b \neq 0$, we have a smooth quadratic surface:
    \[
        W_{2b/3,b}\colon \qquad b^{2} x_{0}^{2} - 4 b x_{0} x_{1} - 2 b x_{0} x_{3} + 3 x_{1}^{2} - 6 x_{1} x_{2} - 9 x_{2}^{2} + 12 x_{2} x_{3} - 3 x_{3}^{2} = 0.
    \]
\end{example}

In view of \Cref{theorem:quadraticsurfaces,theorem:identifiability}, it is natural to ask the following.

\begin{question}
    Given a quadratic surface $W \subset \bbP^3$ which is guaranteed to be Hadamard product of two lines, how to recover two such lines? Recall that they are not uniquely determined by \Cref{remark:stabilizer}.
\end{question}


\section{Surfaces as Hadamard products of higher degree curves}
\label{sec:higher_degrees}

In the following, all the varieties involved in the Hadamard product are assumed to be irreducible.

\begin{notation}
    For every subvarieties $X, Y \subset \bbP^N$, denote by $h_{X,Y}\colon X \times Y \dashrightarrow X \star Y$ the restriction of the Hadamard map $h \colon \bbP^N \times \bbP^N \dashrightarrow\bbP^N$ to $X \times Y$.
\end{notation}

\begin{remark}
\label{remark:morphismcurves}
    Let $C_1, C_2\subset \bbP^3$ be curves. We describe when the Hadamard map $h_{C_1, C_2}$ is a morphism. It is sufficient that $p\star q$ is defined for all $p \in C_1$ and $q \in C_2$. Since both $C_1$ and $C_2$ meet all coordinate planes, i.e.\ contain a point with its $i$-th coordinate equal to $0$ for all $i=0,\ldots,3$, a necessary condition is that neither $C_1$ nor $C_2$ contain one of the 4 coordinate points. If this assumption is satisfied, a sufficient condition is that at least one among $C_1$ and $C_2$ does not meet one of the $6$ coordinate lines in $\bbP^3$. 
    If, for instance, $C_1$ meets the coordinate line $H_i \cap H_j$, then $C_2$ should not meet the orthogonal coordinate line.
\end{remark}

\begin{remark}
\label{remark:notisomorphism}
    Let $X \subset \bbP^N$ be a variety of dimension $0<d<N$ and let $Y \subset \bbP^N$ be a variety of codimension $d$. Since $X \times Y$ is not isomorphic to $\bbP^N$, the Hadamard map $h_{X,Y}$ is not an isomorphism.
\end{remark}

\begin{proposition}
\label{prop:isoiffquadricsurface}
    Let $X \subset \bbP^N$ be a variety of dimension $1\leq d < N/2$ and let $Y \subset \bbP^N$ be a variety of codimension $d+1$. Then, the Hadamard map $h_{X,Y}$ is an isomorphism if and only if $X$ and $Y$ are both lines in $\bbP^3$ and $W = X \star Y$ is a quadratic surface.
\end{proposition}
\begin{proof}
    Suppose that $h_{X,Y}$ is an isomorphism. Then, $\dim W = \dim X + \dim Y = d + N - (d +1) = N - 1$ and hence $W$ is an hypersurface in $\bbP^N$ isomorphic to $X \times Y$, from which $\deg(W) = \deg(X) + \deg(Y) > 1$. Thus $\dim H^0(W,\calO_W(1)) = N+1$. The Hadamard map $h_{X,Y}$ induces an isomorphism in cohomology $\eta\colon V \overset{\sim}{\to} V_1\otimes V_2$, where $V = H^0(W,\calO_W(1))$, $V_1 \otimes V_2 = H^0(X,\calO_X(1)) \otimes H^0(Y, \calO_Y(1)) = H^0(X \times Y, \calO_{X \times Y}(1,1))$. Since $X$ and $Y$ are embedded in projective space, $\dim V_1 \geq d+1$ and $\dim V_2 \geq N - d$ with equality if and only if $X$ and $Y$ are linear spaces. By assumption, $2 \leq d+1 \leq N-d$. Since $V = V_1 \otimes V_2$, $N + 1 = (N-d)+ (d+1) \geq (d+1)(N-d)$, which in turn implies that $d = 1$, $N = 3$ and $\text{codim}(Y) = d+1 = 2$.
    The opposite direction is given by \Cref{lemma:hadamardproduct_smooth}.
\end{proof}

\begin{proposition}
\label{prop:higherdegreesurface}
    Let $L \subset \bbP^3$ be a line and $C \subset \bbP^3$ be a curve. If $W = L \star C$ is a surface of degree $\deg(W) \geq 3$, then $W$ is a cone or $\dim \text{Sing}(W) = 1$.
\end{proposition}
\begin{proof}
    Since $W$ is a high-degree surface, $C$ is not contained in $H_0$, hence $U =C \cap (\bbP^3 \smallsetminus H_0)$ is dense in $C$. In particular, $W$ contains the infinite family of lines $\{ L \star q : q \in U\}$. Hence, either $W$ is a cone or a projective bundle over a smooth curve \cite[Corollary 10.4.5]{Dolgachev_2012}. In the latter case, $W$ is singular along a curve. Indeed, assume that $W$ has at most finitely many singular points. Since $W$ is not a cone, there is a line $\ell$ of the ruling contained in the smooth part of $W$. Since $\ell$ is an element of the ruling, the normal bundle of $\ell$ in $W$ has non-negative degree. From the adjunction formula for $W\subset\bbP^3$, we get that $\omega_W \cong \calO_W(\deg W - 4)$. From the adjunction formula for $\ell \subset W$, $\deg \omega_\ell \geq -1$, a contradiction.
\end{proof}

\begin{corollary}
\label{cor:non-trasversal}
    Let $L \subset \bbP^3$ be a line and $C \subset \bbP^3$ be a curve. If $W = L \star C$ is a surface of degree $\deg(W) \geq 3$ and the Hadamard map $h_{L,C}$ is a morphism, then the intersection $W \cap H_i$ is non-transversal for each $i = 0, \ldots , 3$. In particular, $W$ is not a cone.
\end{corollary}
\begin{proof}
    By \Cref{prop:higherdegreesurface}, we know that $W$ is either a cone or singular along a curve. Take $p_i \in W \cap H_i$ for each $i = 0, \ldots, 3$. Since $h_{L,C}$ is a morphism, 
    $L \star p_i$ is a line contained in $W \cap H_i$ for all $ i = 0, \ldots, 3$. Hence, $W$ intersects $H_i$ non-transversally for all $i = 0, \ldots, 3$. Now, suppose by contradiction that $W$ is a cone and denote by $p \in W$ its vertex. Then, every line in $W$ contains $p$. In particular, $p \in L \star p_i \subset H_i$ for each $i = 0, \ldots , 3$, i.e.\ $p \in H_0 \cap H_1 \cap H_2 \cap H_3 = \emptyset$, a contradiction.
\end{proof}

\begin{example}
    As an example of \Cref{prop:higherdegreesurface}, let $L\subset \bbP^3$ be a generic line through $(0:0:1:1)$ and $C\subset \bbP^3$ be a generic (plane) conic through $(1:1:0:0)$. The result is a cubic surface singular along a line. Direct computation were performed with the support of the algebra software Macaulay2 \cite{M2}, see \cite{datahadamardquadric} for the code.
\end{example}

\bibliographystyle{alphaurl}
\bibliography{hadaquadric_v2.bib}

\end{document}